\documentclass[12pt,leqno]{amsart}
\usepackage{amscd,times,fullpage}
\usepackage{ucs}
\usepackage[utf8x]{inputenc}

\usepackage{amsthm,amsmath}
\usepackage{amssymb}
\usepackage{graphicx}
\usepackage{amsthm}
\usepackage{enumerate}
\usepackage{tikz-cd}
\usetikzlibrary{cd}
\usetikzlibrary{positioning}
\usetikzlibrary{matrix}

\newcommand{\lra}{\longrightarrow}             
\newcommand{\R}{\mathbb{R}}

\newcommand{\Q}{\mathbb{Q}}
\newcommand{\CP}{\mathbb{C}\mathrm{P}}

\newcommand{\Sec}{\mathrm{Sec}}
\newcommand{\Ric}{\mathrm{Ric}}

\newcommand{\C}{\mathbb{C}}            
\newcommand{\de}{\partial}          

\newcommand{\K}{K\"{a}hler }

\newcommand{\ov}[1]{\overline{#1}}

\newcommand{\deb}{\ov\partial}

\newcommand{\di}{{\operatorname{d}}}

\newcommand{\Id}{\operatorname{Id}}

\newcommand{\F}{\mathcal{F}}

\newcommand{\D}{\mathcal{D}}

\newcommand{\Lie}{\mathcal{L}}

\newtheorem{theor}{Theorem}

\newtheorem{defin}{Definition}[section]

\newtheorem{cor}{Corollary}

\newtheorem{remark}{Remark}

\newtheoremstyle{colon}%
{}
{}
{\itshape}
{}
{\bfseries}
{:}
{ }
{}

\theoremstyle{colon}
\newtheorem*{question}{Question}

\begin{document}

\title[Sasakian immersions of Sasaki-Ricci solitons into Sasakian space forms]{Sasakian immersions of Sasaki-Ricci solitons into Sasakian space forms}

\author{G.~Placini}
\address{Dipartimento di Matematica e Informatica, Universit\'a degli studi di Cagliari, Via Ospedale 72, 09124 Cagliari, Italy}
\email{giovanni.placini@unica.it}

\date{\today ; {\copyright \ G.~Placini 2020}}


\begin{abstract}
Let $(g,X)$ be a Sasaki-Ricci soliton on a Sasakian manifold $S$. 
We prove that if $(S,g)$ admits a local Sasakian immersion in a Sasakian space form $S(N,c)$ of constant $\phi$-sectional curvature $c$, then $S$ is $\eta$-Einstein and its $\eta$-Einstein constants are rational. 
Moreover, if $c\leq -3$, $S$ is locally equivalent to the Sasakian space form $S(n,c)$ and its $\eta$-Einstein constants are determined by $c$. 
Further results are obtained in the compact setting, i.e. when $c>-3$, under additional hypotheses.
\end{abstract}
 
\maketitle


\section{Introduction and statements of the main results}\label{sectionint}

Since the publication of \cite{boyergalicki08} Sasakian geometry has received growing interested. This is possibly due to the connections with physics and the abundance of structures that concur in a Sasakian manifold. This abundance allows one to study Sasakian geometry from several points of view.
A Sasakian manifold $S$ is a contact manifold endowed with a compatible metric $g$ and holomorphic structure transverse to the Reeb foliation. The cases where the metric satisfies some additional property are widely studied. An instance of such properties is given by Sasaki-Einstein metrics. Sasaki-Einstein manifolds have drown the attention of many mathematicians in the last few decades. A generalization of Sasaki-Einstein metrics is represented by $\eta$-Einstein metrics, i.e., transversally Einstein metrics. Namely, Sasakian manifolds that satisfy the relation
$$
\Ric_g=\lambda g+\nu \eta\otimes\eta
$$
for some constants $\lambda,\nu \in\R$, where $\Ric_g$ is the Ricci tensor of the metric and $\eta$ denotes the contact form. 
We refer the reader to \cite{boyergalicki08,boyergalickimatzeu06} for an introduction. A further generalization of such metrics is given by Sasaki-Ricci solitons (SRS for short). These have been introduced in \cite{futakionowang09} as special solutions of the Sasaki-Ricci flow of \cite{smoczykwangzhang10}. Since then SRS have received growing interest, see for instance \cite{petrecca16,tadano15}.

Sasakian geometry seats between two K\"ahler geometries. Namely, the K\"ahler structure of the cone and the one transverse to the Reeb foliation. For this reason it is often referred to as the odd dimensional counterpart to K\"ahler geometry. It is then natural to try to extend known results for K\"ahler manifolds to the Sasakian setting. Such an instance is given by the problem of immersing certain classes of K\"ahler manifolds into complex space forms. This translates in the Sasakian setting to
\begin{question}
Which Sasakian manifolds admit a Sasakian immersion into a Sasakian space form?
\end{question}
The question above is clearly too general so that it needs to be specialized to specific classes in order to be answered.
For instance the problem has been studied when the immersion is taken to be CR, i.e. to preserve the underlying CR structure but not necessarily the metric. In fact, Ornea and Verbitsky \cite{orneaverbitsky07} proved that any compact Sasakian manifold admits a CR embedding in a Sasakian manifold diffeomorphic to a sphere.
 The case of $\eta$-Einstein manifolds admitting a Sasakian immersion has been, in some instances only partially, solved in \cite{bandecappellettimontanoloi20,cappellettimontanoloi19,kenmotsu69}. On the other hand, this question has not been addressed for Sasaki-Ricci solitons. Our goal in this article is to characterize SRS that admit a (local) immersion into a Sasakian space form.

In light of the K\"ahler/Sasaki parallel, our results extend those of \cite{loimossa20} to the Sasakian setting.
The first of these is the following

\begin{theor}\label{ThmMain1}
Let  $S$  be a $(2n+1)$-dimensional complete regular Sasakian manifold endowed with a Sasaki-Ricci soliton. 
Suppose there exists a neighbourhood $U_p$ of a point $p\in S$ and an immersion $\psi\colon U_p\lra S(N,c)$ into a Sasakian space form $S(N,c)$ with $c\leq -3$ .

 Then $S$ is Sasaki equivalent to $S(n,c)/\Gamma$ for a discrete subgroup $\Gamma$ of the group of Sasakian tranformations of $S(n,c)$. In particular, $S$ is $\eta$-Einstein and its $\eta$-Einstein constants $(\lambda,\nu)$ are determined by $c$.

If additionally $U_p=S$, then $\Gamma=0$ and, up to a Sasakian transformation of $S(n,c)$, $\psi$ is of the form $$\psi(z,t)=(z,0,t+a)$$ for a constant $a$.
\end{theor}

In contrast with the K\"ahler analogue of Theorem~\ref{ThmMain1} proved in \cite{loimossa20}, we are not assuming the existence of a global immersion of $S$ into a Sasakian space form. 
Furthermore, while in \cite{loimossa20} the Einstein constant is a rational multiple of the holomorphic curvature, in the Sasakian setting the Einstein constants are completely determined by the $\phi$-sectional curvature $c$, cf. \eqref{EqSasakiRicciHolo}.
Namely, we have 
$$\Ric_{g}=\left( \dfrac{n+1}{2}(c+3)-2\right) g+\left( 2n+2-\dfrac{n+1}{2}(c+3) \right)\eta\otimes\eta\,\, .$$

Our second result addresses the problem in the case where the Sasakian space form is compact, i.e. the standard sphere $S^{2N+1}$.
Very little is known in this setting, compared to the non-compact case, even in the Sasakian manifold is assumed to be $eta$-Einstein, see \cite{cappellettimontanoloi19}. For this reason we state a general result and several corollaries under additional or different hypotheses. 

\begin{theor}\label{ThmMain2}
Let  $S$  be a $(2n+1)$-dimensional complete regular Sasakian manifold endowed with a Sasaki-Ricci soliton. 
	Suppose there exists a neighbourhood $U_p$ of a point $p\in S$ and an immersion $\psi\colon U_p\lra S^{2N+1}$ into the standard Sasakian sphere.
    Then $(S,g)$ is a $\eta$-Einstein Sasakian manifold whose $\eta$-Einstein constants $(\lambda,\nu)$ are given by $\lambda=4\mu -2$ for some $\mu\in\Q$.
\end{theor} 
Theorem~\ref{ThmMain1} and Theorem~\ref{ThmMain2} have interesting consequences when investigating which SRS admit a global immersion into a Sasakian space forms. 
It is natural to consider first $\eta$-Einstein Sasakian manifolds, as these are trivially SRS.
Combining the results in the compact and non-compact cases we see that the existence of an immersion into certain Sasakian space forms constrains the possible values of the $\eta$-Einstein constants $(\lambda,\nu)$.
\begin{cor}\label{Cor1}
	A complete, $\eta$-Einstein manifold with rational $\eta$-Einstein constants cannot be immersed in a Sasakian space form $S(N,c)$ of irrational $\phi$-sectional curvature $c\in\R\setminus\Q$.
\end{cor}
The assumptions of completeness ensures that the leaf space is a manifold, i.e., that the Sasakian manifold $S$ fibers over a K\"ahler manifold.
Next we concentrate on some consequences of Theorem~\ref{ThmMain2}. Firstly, let us consider the case where $S$ is compact and the codimension of the immersion is arbitrary, cf. Theorems $1$ and $3$ in \cite{cappellettimontanoloi19}.

\begin{cor}\label{Cor2}
    Let  $S$  be a $(2n+1)$-dimensional compact Sasakian manifold endowed with a Sasaki-Ricci soliton. 
	Suppose there exists an immersion $\psi\colon S\lra S^{2N+1}$ into the standard Sasakian sphere.
	Then $(S,g)$ is either a $\eta$-Einstein Sasakian manifold with rational $\eta$-Einstein constants $(\lambda>2n,\nu)$ or it is Sasaki equivalent to $S^{2n+1}$.
\end{cor}
Notice that in this case $S$ is necessarily regular since it can be immersed into a regular Sasakian manifold.
Lastly we turn our attention to the case of small codimension to prove the two following corollaries, cf. part $(i)$ of the main theorem of \cite{kenmotsu69} and \cite[Theorem~2]{cappellettimontanoloi19}.

\begin{cor}\label{Cor3}
Let  $S$  be a $(2n+1)$-dimensional complete (not necessarily compact) Sasakian manifold endowed with a Sasaki-Ricci soliton. 
	Suppose there exists an immersion $\psi\colon S\lra S^{2n+3}$ into the standard Sasakian sphere.
	Then $(S,g)$ is Sasaki equivalent to $S^{2n+1}$ or the Boothby-Wang bundle over the complex quadric $Q_n\in\CP^{n+1}$.
\end{cor}

\begin{cor}\label{Cor4}
Let  $S$  be a $(2n+1)$-dimensional compact Sasakian manifold endowed with a Sasaki-Ricci soliton. 
	Suppose there exists an immersion $\psi\colon S\lra S^{2n+5}$ into the standard Sasakian sphere.
	Then $(S,g)$ is Sasaki equivalent to $S^{2n+1}$ or the Boothby-Wang bundle over the complex quadric $Q_n\in\CP^{n+1}$.
\end{cor}

We believe that these results can be generalized to arbitrary codimension.
Such a generalization of Corollary~\ref{Cor3} and Corollary~\ref{Cor4} is related to an analogous conjecture on Sasakian immersions of $\eta$-Einstein manifolds, see \cite{cappellettimontanoloi19} and referecences therein.

\vspace{1cm}


\section{Sasakian manifolds and Sasaki-Ricci solitons}\label{sectionbackground}

\subsection{Sasakian manifolds and immersions}

We begin with some definitions and known results, for a more exhaustive treatment we refer to the monograph by Boyer and Galicki~\cite{boyergalicki08}. 
All manifolds are assumed to be smooth, connected and oriented.

A \textit{K-contact structure} $(S,\eta,\phi,R,g)$ on a manifold $S$ consists of a contact form $\eta$ and an endomorphism $\phi$ of the tangent bundle $TS$ satisfying the following properties:
\begin{enumerate}
	\item[$\bullet$] $\phi^2=-\Id+R\otimes\eta$ where $R$ is the Reeb vector field of $\eta$,
	\item[$\bullet$] $\phi_{\vert\D}$ is an almost complex structure compatible with the symplectic form $\di\eta$ on $\D=\ker\eta$,
	\item[$\bullet$] The Reeb vector field $R$ is Killing with respect to the metric $g(\cdot,\cdot)=\di\eta(\phi\cdot,\cdot)+\eta(\cdot)\eta(\cdot)$.
\end{enumerate}   
Given such a structure one can consider the almost complex structure $I$ on the Riemannian cone $\big( S\times(0,\infty),t^2g+\di t^2\big)$ given by
\begin{enumerate}
	\item[$\bullet$] $I=\phi$ on $\D=\ker\eta$,
	\item[$\bullet$] $I(R)=t\partial_t$.
\end{enumerate} 
A \textit{Sasakian structure} is a K-contact structure $(S,\eta,\phi,R,g)$ such that the associated almost complex structure $I$ is integrable. 
We call a manifold $S$ \textit{Sasakian} if it is equipped with a Sasakian structure.

A Sasakian manifold is called \textit{regular} (respectively \textit{quasi-regular}, \textit{irregular}) if its Reeb foliation is such.
Every regular compact Sasakian manifold is a \textit{Boothby-Wang fibration} $S$ over a projective manifold $(K,\omega)$ with $\omega$ representing an integral class (\cite{boothbywang58,boyergalicki08}), that is, the principal $S^1$-bundle $\pi\colon S\longrightarrow K$ with Euler class $[\omega]$ and connection $1$-form $\eta$ such that $\pi^*(\omega)=\di\eta$.
This is not necessarily true in the non-compact case. Nevertheless one can prove a similar statement under some additional conditions, cf. the proof of Theorem~\ref{ThmMain1}.
In general the Reeb foliation $\F$ of a Sasakian structure is transversally K\"ahler. This endows the space of leaves with a K\"ahler structure.

Two Sasakian manifolds $(S_1,\eta_1,\phi_1,R_1,g_1)$ and  $(S_2,\eta_2,\phi_2,R_2,g_2)$ are \text{equivalent} if there exists a diffeomorphism $f\colon S_1\lra S_2$ such that
\begin{equation*}
f^*\eta_2=\eta_1\hspace{5mm}\mbox{and}\hspace{5mm}f^*g_2=g_1.
\end{equation*}
If this holds then $f$ also preserves the endomorphism $\phi_1$ and the Reeb vector field.
 A Sasakian equivalence from a Sasakian manifold $(S,\eta,\phi,R,g)$ to itself is called a \textit{Sasakian transformation} of $(S,\eta,\phi,R,g)$.
 
 In this article we discuss Sasakian immersions into Sasakian space forms.
Given two Sasakian manifolds$(S_1,\eta_1,\phi_1,R_1,g_1)$ and  $(S_2,\eta_2,\phi_2,R_2,g_2)$, a \textit{Sasakian immersion} of $S_1$ in $S_2$ is an immersion $\psi\colon S_1\lra S_2$ such that

\begin{align*}
\psi ^*\eta_2&=\eta_1,\hspace{20mm}\psi ^*g_2=g_1,\\
\psi_*R_1&=R_2\hspace{5mm}\mbox{and}\hspace{5mm}\psi_*\circ \phi_1=\phi_2\circ\psi_*.
\end{align*}

\subsection{Sasakian $\eta$-Einstein manifolds and Sasaki-Ricci solitons}

The Riemannian properties of Sasakian manifold, in particular Sasaki-Einstein and $\eta$-Einstein metrics, have received great attention from many authors, partially due to their connection to physics. 
We recall now the definitions and main properties of these structures with a particular focus on the relation with the transverse K\"ahler geometry, we refer the interested reader to \cite{boyergalicki08,boyergalickimatzeu06}

On a Sasakian manifold $(S,\eta,\phi,R,g)$ the tangent bundle splits canonically as $TS=\D\oplus T_\F$ where $\D=\ker\eta$ and $T_\F$ denotes the tangent to the Reeb foliation $\F$. The transverse K\"ahler geometry is given by $(\D,\phi_{\vert_\D},\di\eta)$.
When the space of leaves of the Reeb foliation is a K\"ahler manifold $(K,J,\omega)$ we have a fibration $\pi\colon S\lra K$ such that
\begin{equation*}
\pi^*\omega=\di\eta\hspace{5mm}\mbox{and}\hspace{5mm}\pi_*\circ\phi=J\circ\pi_*.
\end{equation*}
In virtue of this the metric decomposes as 
\begin{equation}\label{EqMetricSplit}
g=g^T\oplus\eta\otimes\eta
\end{equation}
where $g^T(\cdot,\cdot)=\di\eta(\cdot,\phi\cdot)$.
With an abuse of notation we write $g^T$ for both the transverse metric and the metric on $K$.
It follows from \eqref{EqMetricSplit} that the Riemannian properties of $S$ can be expressed in terms of those of the transverse K\"ahler geometry and of the contact form $\eta$. For instance, the Ricci tensor of $g$ is given by
\begin{equation}\label{EqRicciSplit}
\Ric_g=\Ric_{g^T}-2g.
\end{equation}
A Sasakian manifold $(S,\eta,\phi,R,g)$ is said to be \textit{$\eta$-Einstein} if the Ricci tensor satisfies
\begin{equation}\label{EqEtaEinstein}
\Ric_g=\lambda g+\nu \eta\otimes\eta
\end{equation}
for some constants $\lambda,\nu \in\R$. 
It follows from \eqref{EqRicciSplit} and \eqref{EqEtaEinstein} that a Sasakian manifold is $\eta$-Einstein with constants $(\lambda,\nu)$ if, and only if, its transverse geometry is K\"ahler-Einstein with Einstein constant $\lambda+2$.
Since on a K-contact manifold the Ricci tensor satisfies $\Ric(R,X)=2n\eta(X)$ we have that $\lambda+\nu=2n$.
Therefore, a Sasakian $\eta$-Einstein manifold is Einstein if and only if $\lambda=2n$, that is, if it has a transverse K\"ahler-Einstein geometry with Einstein constant $2n+2$.

Generalizing $\eta$-Einstein manifolds are Sasaki-Ricci solitons. In order to introduce them we need to recall known facts about the transverse K\"ahler geometry.
On a Sasakian manifold $S$ of dimension $2n+1$ there exists a covering $\{U_\alpha\}$ with foliated charts $\varphi_\alpha\colon U_\alpha\lra\varphi_\alpha(U_\alpha)\subset\R\times\C^n$. Denote by $\pi_\alpha$ the following map $$\pi_\alpha=\pi_{\C^n}\circ\varphi_\alpha\colon U_\alpha\lra V_\alpha\subset\C^n\,\, .$$
The Sasakian structure is transversally holomorphic, that is, the maps 
$\pi_\alpha\circ\pi_\beta^{-1}\colon V_\alpha\cap V_\beta\lra V_\alpha\cap V_\beta$ are biholomorphisms.
A \textit{basic $p$-form} $\alpha$ on $S$ is a $p$-form such that
$$\iota_R\alpha=0,\hspace{5mm}\Lie_R\alpha=0\, .$$
It is easy to see that the exterior derivative $\di$ sends basic forms to basic forms. Therefore we denote it by $\di_B$ when we want to emphasize that it is restricted to basic forms.
Suppose now that $(x,z_1,\ldots,z_n)$ are local coordinates in $U_\alpha$. If a basic form $\alpha$ can be written locally as 
$$\alpha=a_{i_1,\ldots,i_{p+q}}\di z_{i_1}\wedge\cdots\wedge\di z_{i_p}\wedge\di\bar{z}_{i_{p+1}}\wedge\cdots\wedge\di\bar{z}_{i_{p+q}},$$
then $\alpha$ is said to be a basic $(p,q)$-form. One can show that such a local form is also of type $(p,q)$ in any chart $U_\beta$ with $U_\alpha\cap U_\beta\neq\emptyset$.
Therefore we have  well defined operators $\de_B$, respectively $\deb_B$, of degree $(1,0)$, resp. $(0,1)$, such that $\di_B=\de_B+\deb_B$.
\begin{defin}
	A complex vector field $X$ on a Sasakian manifold $S$ is called Hamiltonian holomorphic if it satisfies the following conditions
	\begin{enumerate}[(a)]
		\item the vector field $\di\pi_\alpha(X)$ is holomorphic on $V_\alpha$,
		\item the function $u_X\colon=i\eta(X)$ is such that $\deb u_X=-i\iota_X\di\eta$. 
	\end{enumerate}
\end{defin}

In \cite{smoczykwangzhang10} Smoczyk, Wang and Zhang introduced the Sasaki-Ricci flow
\begin{equation}\label{EqSasRicFlow}
\frac{\di}{\di t}g^T(t)=-\left( \Ric_{g^T(t)} -\lambda g^T(t)\right)
\end{equation}
with the aim of proving the existence of $\eta$-Einstein metrics.
In order to study the Sasaki-Ricci flow on positive Sasakian manifolds, Futaki, Ono and Wang \cite{futakionowang09} defined Sasaki-Ricci solitons as a pair $(g,X)$ consisting of a Sasakian metric $g$ and a Hamiltonian holomorphic vector field $X$ such that
\begin{equation}\label{EqWrongSoliton}
\Ric_{g^T}=(2n+2)g^T+\Lie_Xg^T.
\end{equation}
The equation \eqref{EqWrongSoliton} is commonly used in literature to define Sasaki-Ricci solitons, see for instance \cite{petrecca16,tadano15}.
On the other hand, on a complex manifold $M$ a K\"ahler-Ricci soliton (KRS for short) is defined to be a pair $(g,X)$  where $g$ is a K\"ahler metric and $X$ is a holomorphic vector field on $M$ satisfying
\begin{equation}\label{EqKahlerSoliton}
\Ric_{g}=\lambda g+\Lie_Xg
\end{equation}
with $\lambda\in \R$. Although a compact K\"ahler manifold does not admit non-trivial KRS with $\lambda\leq 0$, these cases are widely studied on open K\"ahler manifolds. 
In analogy with the K\"ahler setting we generalize the above definition and give the following
\begin{defin}\label{DefSasakiSoliton}
A \textbf{Sasaki-Ricci soliton} (SRS in short) on a Sasakian manifold $S$ is a pair $(g,X)$ consisting of the Sasakian metric $g$ and a Hamiltonian holomorphic vector field $X$ such that
\begin{equation}
\Ric_{g^T}=\lambda g^T+\Lie_Xg^T
\end{equation}
for some $\lambda\in\R$.
\end{defin}
If a manifold $S$ is endowed wit a SRS, with an abuse of notation we will  simply say that $S$ is a SRS. 
One can easily construct examples of SRS on open Sasakian manifolds, also in the case where $\lambda\leq0$, as bundles over certain gradient KRS, see for instance \cite{cao96,cao97,feldmanilmanenknopf03}.

\begin{remark}\label{RmkInducedSoliton}
	By definition a Sasaki-Ricci soliton $(X,g)$ on a regular Sasakian manifold is a KRS on the transverse K\"ahler geometry. In particular, if the space of leaves of the Reeb foliation is a smooth manifold $K$, then it has a canonically induced KRS.
\end{remark}

\subsection{Sasakian space forms}
Let  $(S,\eta,\phi,R,g)$ be a Sasakian manifold.
If $\Sec$ is the ordinary Riemannian sectional curvature of $g$, then the \textit{$\phi$-sectional curvature $H$}
of $g$ is defined by 
$$H(X)=\Sec(X,\phi X)$$
for all vector fields $X$ of unit length orthogonal to $R$.

A \textit{Sasakian space form} $S(n,c)$ is a Sasakian manifold of dimension $2n+1$ with constant $\phi$-sectional curvature $H\equiv c$.
Tanno \cite{tanno69} proved that there are three types of Sasakian space forms, namely, those with $H\equiv c<-3,=-3$ and $>-3$.
These are analogous to complex space forms, that is, complex manifolds of constant holomorphic sectional curvature. Indeed, under the Boothby-Wang correspondence, constant $\phi$-sectional curvature $c$ corresponds precisely to constant holomorphic transverse sectional curvature $c+3$.

Explicitly Tanno proved that every Sasakian space form is a quotient of one of the following three by a subgroup of Sasakian transformations.
\begin{itemize}
	\item If $c>-3$, $S(n,c)$ is Sasaki equivalent to the Sasakian sphere $S^{2n+1}(c)$. This is the Boothby-Wang bundle over $\CP^n(c+3)$.
	\item If $c=-3$, $S(n,c)$ is Sasaki equivalent to $\R^{2n+1}(-3)$. The transverse K\"ahler structure is the standard one on $\C^n$.
	\item If $c<-3$, $S(n,c)$ is Sasaki equivalent to $B^{2n+1}_\C(c+3)\times\R$ where the transverse K\"ahler structure is that of the hyperbolic complex space $B^{2n+1}_\C(c+3)$ of constant holomorphic sectional curvature $c+3$.
\end{itemize}

Finally we make note of the Ricci tensor for the transverse K\"ahler structures of constant holomorphic sectional curvature $c+3$:
\begin{equation}\label{EqKahlerRicciHolo}
\Ric_{g^T}=\dfrac{n+1}{2}(c+3)g^T
\end{equation}
which in turn implies
\begin{equation}\label{EqSasakiRicciHolo}
\Ric_{g}=\left( \dfrac{n+1}{2}(c+3)-2\right) g+\left( 2n+2-\dfrac{n+1}{2}(c+3) \right)\eta\otimes\eta\,\, .
\end{equation}

\section{Proof of the main results}\label{sectionproofs}
\begin{proof}[Proof of Theorem~\ref{ThmMain1}]
Let $(X,g)$ be a Sasaki-Ricci soliton on a Sasakian manifold $S$. Suppose there exists a neighbourhood $U_p$ of a point $p\in S$ and an immersion $\psi\colon U_p\lra S(N,c)$  into a Sasakian space form $S(N,c)$ with $c\leq -3$. 

We cannot conclude that $S$ is an $S^1$-bundle over a K\"ahler manifold because $S$ is not necessarily compact.
Nevertheless, the Reeb foliation still defines a fibration $\pi\colon S\lra K$ over a K\"ahler manifold because $S$ is regular and complete, see \cite{reinhart59}.

Now $\psi$ covers a K\"ahler immersion into a definite complex space form $K(N,c+3)$ of dimension $2N$ and constant holomorphic curvature $c+3$, see \cite{cappellettimontanoloi19,harada72}.
Thus we get the commutative diagram
$$
\begin{tikzcd}
U_p \arrow[r, "\psi"] \arrow[d, "\pi"'] & S(N,c)  \arrow[d, "\pi'"] \\
V_x  \arrow[r, "\phi"'] &   K(N,c+3)
\end{tikzcd}
$$
where $x=\pi(p)$ and $V_x=\pi(U_p)\subset K$.

Moreover, the space of leaves $K$ of the Reeb fibration is a K\"ahler manifold equipped with a K\"ahler--Ricci soliton $(\di\pi(X),g^T)$, cf. Remark~\ref{RmkInducedSoliton}.
The existence of the Ricci soliton $(g^T,\di \pi(X))$ implies that the K\"ahler metric $g^T$ is real-analytic, (see \cite[Corollary~1.3]{kotschwar13}),.
Therefore $K$ is a complex manifold equipped with a real-analytic K\"ahler metric which admits a local immersion $V_p\lra K(N,c+3)$ into a complex space form. 
Then a classical result of Calabi \cite{calabi53} implies that for every point $y\in K$ there exists a neighbourhood $V_y$ and a K\"ahler immersion $V_y\lra K(N,c+3)$.

Hence, the main result of \cite{loimossa20} implies that $(V_y,g^T)$ is K\"ahler--Einstein.
As reviewed in Section~\ref{sectionbackground}, this is equivalent to $\pi^{-1}(V_y)$ being $\eta$-Einstein.
Now $(S,g)$ is $\eta$-Einstein because the sets $\pi^{-1}(V_y)$ cover $S$.

Therefore the thesis of Theorem~\ref{ThmMain1} follows from a result of Bande, Cappelletti--Montano and Loi \cite{bandecappellettimontanoloi20} on immersions of $\eta$-Einstein manifolds into Sasakian space forms.
\end{proof}

\begin{proof}[Proof of Theorem~\ref{ThmMain2}]
	Let $(X,g)$ be a Sasaki-Ricci soliton on a Sasakian manifold $S$. Suppose there exists a neighbourhood $U_p$ of a point $p\in S$ and an immersion $\psi\colon U_p\lra S^{2N+1}$  into the standard Sasakian sphere. 
	
	Following the arguments in the proof of Theorem~\ref{ThmMain1} we see that $S$ is the total space of a fibration $\pi\colon S\lra K$ over a K\"ahler manifold $K$.
	Moreover,  $\psi$ covers a K\"ahler immersion
	
	$$
	\begin{tikzcd}
	U_p \arrow[r, "\psi"] \arrow[d, "\pi"'] & S^{2N+1}  \arrow[d, "\pi'"] \\
	V_x \arrow[r, "\phi"'] &   \CP^N
	\end{tikzcd}
	$$
	where $V_x=\pi(U_p)$ and $\pi'$ is the standard Hopf bundle.
	
	Hence, every point of $K$ admits a neighbourhood which can be immersed in $\CP^N$. This implies that $K$ is a \K-Einstein manifold with Einstein constant $\lambda+2=4\mu$ for some rational number $\mu$. We conclude that $S$ is $\eta$-Einstein with constants $(\lambda,\nu)$ given by $\lambda=4\mu-2$.
\end{proof}

\begin{proof}[Proof of Corollary~\ref{Cor2}]
Since $S$ is admits a global Sasakian immersion into a regular Sasakian manifold, it is itself regular, cf. \cite[Proposition~1]{cappellettimontanoloi19}.
Thus $S$ is a Boothby-Wang bundle $\pi\colon S\lra K$ over a compact K\"ahler manifold $K$ which is endowed with the induced KRS.

Moreover, $K$ admits a K\"ahler immersion in $\CP^N$ which is covered by $\psi$:
	$$
\begin{tikzcd}
S \arrow[r, "\psi"] \arrow[d, "\pi"'] & S^{2N+1}  \arrow[d, "\pi'"] \\
K \arrow[r, "\phi"'] &   \CP^N\,\,\, .
\end{tikzcd}
$$
Again by \cite{loimossa20} $K$ is \K-Einstein with rational Einstein constant $\lambda+2$.
Therefore $S$ is a $\eta$-Einstein manifold with constants $(\lambda,\nu)$ which admits a Sasakian immersion in $S^{2N+1}$.

Now \cite[Theorem~3]{cappellettimontanoloi19} implies that $\lambda\geq 2n$.
Moreover, if $\lambda=2n$, then $S$ is Sasaki-Einstein and \cite[Theorem~1]{cappellettimontanoloi19} implies that $S$ is Sasakian equivalent to $S^{2n+1}$.
\end{proof}

\begin{remark}
	With the definition of SRS given in \cite{futakionowang09}  Theorem~\ref{ThmMain1} and Corollary~\ref{Cor2} imply that a compact SRS immersed in a Sasakian space form is necessarily Sasakian equivalent to a standard sphere $S^{2n+1}$.
\end{remark}

\begin{proof}[Proof of Corollary~\ref{Cor3} and Corollary~\ref{Cor4}]
	Notice that the Sasakian structure on $S$ is regular becuse it admits an immersion into a regular Sasakian manifold.
	
   We want to prove now that $S$ fibers over a K\"ahler manifold $K$. As discussed in the proofs above this follows from the classical result of Boothby and Wang \cite{boothbywang58} when $S$ is compact and from \cite{reinhart59}  when $S$ is complete.

 The same line of arguments as in the proof of Theorem~\ref{ThmMain1} shows that $S$ is $\eta$-Einstein.
 Now Corollary~\ref{Cor3} follows from the main result of \cite{kenmotsu69} while Corollary~\ref{Cor4} is a consequence of \cite[Theorem~2]{cappellettimontanoloi19}
\end{proof}

\end{document}